\documentclass[12pt]{article}
\usepackage[english]{babel}
\usepackage{amsmath}
\usepackage{amsthm}
\usepackage{amssymb,amsfonts}
\usepackage{bbm}
\usepackage{url}
\usepackage{dsfont}
\usepackage{enumitem}
\usepackage[margin=2.5cm]{geometry}
\usepackage{caption}
\usepackage{subcaption}
\usepackage{graphicx}
\usepackage{adjustbox}
\usepackage[hidelinks]{hyperref}

\usepackage[dvipsnames]{xcolor}

\newtheorem{assumption}{Assumption}
\theoremstyle{definition}

\newcommand{\footremember}[2]{%
    \footnote{#2}
    \newcounter{#1}
    \setcounter{#1}{\value{footnote}}%
}
\newcommand{\footrecall}[1]{%
    \footnotemark[\value{#1}]%
}

\newtheorem{theorem}{Theorem}
\newtheorem{proposition}{Proposition}
\newtheorem{corollary}{Corollary}
\newtheorem{lemma}{Lemma}

\theoremstyle{remark}
\newtheorem{remark}{Remark}[section]

\usepackage{xcolor}

\newcommand{\R}{\mathbb{R}} 

\usepackage[nottoc,notlot,notlof]{tocbibind}
\usepackage{booktabs}

\title{
\normalfont \normalsize   
\huge Strong log-concavity in probit regression 
}
\author{Martin Chak\footremember{alley}{Department of Decision Sciences and BIDSA, Bocconi University}\footnote{martin.chak@unibocconi.it} 
\and Zoraida F. Rico\footrecall{alley} \footnote{zoraida.rico@unibocconi.it}
\and Giacomo Zanella\footrecall{alley} \footnote{giacomo.zanella@unibocconi.it}}
\date{\normalsize\today}

\begin{document}
\maketitle

\begin{abstract}
We show that strong log-concavity emerges in probit regression likelihoods without ridge penalization (i.e.\ Gaussian priors), unlike for the logistic case. 
Specifically, we provide: (a) a non-asymptotic characterization of strong log-concavity for fixed designs, similar to that for the existence of the maximum likelihood estimator (MLE) and (b) an asymptotic analysis for random designs, in the proportional regime when the sample size~$n$ and the number of covariates~$d$ grow proportionally~$d/n\rightarrow r\in[0,1)$. 
In the latter case we show that, provided $r$ is small enough, the resulting condition number is finite and bounded independently of~$n,d,r$ with high-probability. Numerically tractable estimates are given in case~$r=0$. Thus, probit regression provides a non-trivial example of a high-dimensional, well-conditioned log-concave objective.
\end{abstract}

\section{Introduction}
Probit regression~\cite{MR3223057} is a classical model for the relationship between binary responses~$(y_i)_{i=1}^n\in\{-1,1\}^n$ and explanatory covariates~$(\bar{x}_i)_{i=1}^n 
\in(\mathbb{R}^{d+1})^n$. 
The model posits the existence of~$\bar{\beta}\in\mathbb{R}^{d+1}$ such that~$(y_i+1)/2$ has the conditional distribution~$\textrm{Bernoulli}(\Phi(\bar{x}_i\cdot \bar{\beta}))$ given~$\bar{x}_i$ for all~$i$, as well as conditional independence in~$(y_i)_{i=1}^n$ given~$(\bar{x}_i)_{i=1}^n$. 
Inference on~$\bar{\beta}$ then falls to finding a minimizer of the corresponding negative log-likelihood function~$l:\mathbb{R}^{d+1}\rightarrow(0,\infty)$ 
given by~$\theta\mapsto l(\theta) = -\sum_{i=1}^n\log \Phi(y_i\bar{x}_i\cdot\theta)$, in the frequentist approach~\cite{Demidenko2001,pmlr-v151-m,SchumannTripathi2018}, or estimating a conditional distribution~$e^{-l}\cdot \pi$ in the Bayesian approach for a prior density~$\pi$ on~$\mathbb{R}^{d+1}$~\cite{MR1224394,JMLR:v27:25-2192}. In both approaches, strong convexity of~$l$ is a desirable property from a computational viewpoint. In particular, it implies a finite condition number~$L/m<\infty$ 
(see~\eqref{probit},~\eqref{eq:def_m_L} below for definitions of~$L,m$), which yields favourable and canonical conditions to study non-asymptotic convergence of optimization algorithms~\cite{MR2142598} in the frequentist approach, and sampling ones~\cite{Chewi26Book} in the Bayesian approach. 
Motivated by this, we study conditions under which~$l$ is strongly convex and the condition number~$L/m$ can be controlled, which can also be of interest beyond the convergence analysis of sampling and optimization algorithms.

In Section \ref{exs}, we show that 
strong convexity of $l$ is equivalent to the condition
\begin{equation}\label{intcon}
\textrm{Span}((\bar{x}_i)_{i\in I_u})=\mathbb{R}^{d+1} \quad\forall\,u\in\mathbb{S}^d,
\end{equation}
where~$I_u = \{i\in[1,n]\cap\mathbb{N}:y_i\bar{x}_i\cdot u\leq 0\}$, 
which is a strengthening of the necessary and sufficient condition for existence of the MLE provided in~\cite{MR1185200}; see Theorem~\ref{aslp} for a precise statement and Remark~\ref{re0} for the comparison with~\cite{MR1185200}. 
Moreover, in the well-specified case (i.e.\ when~$(y_i)_{i=1}^{\infty}$ is generated according to the model with a given~$\bar{\beta}$) under a Gaussian design assumption (Assumption~\ref{as1} below), we show 
that if~$(d+1)/n$ is small enough, more specifically if 
\begin{equation}\label{dncon}
\frac{d+1}{n} + \bigg(\frac{2(d+1)}{n}\log\bigg(\frac{en}{d+1}\bigg)\bigg)^{1/2} < \mathbb{E}\big[\Phi\big(-\big|\beta_0 + |\beta|Z\big|\big)\big],
\end{equation}
where~$\beta_0\in\mathbb{R}$,~$\beta\in\mathbb{R}^d$ satisfy~$\bar{\beta} = (\beta_0,\beta^{\top})^{\top}$ and~$Z\sim N(0,1)$, 
then condition~\eqref{intcon} is satisfied with exponentially high (in-$n$) probability; 
see Corollary~\ref{cor1} for a precise statement. Note that the left-hand side of~\eqref{dncon} is a strictly increasing function in~$(d+1)/n\in(0,1)$, whilst the right-hand side is a numerically tractable scalar quantity. 
While the arguments in our proofs differ, condition~\eqref{dncon} is similar in nature to the numerical conditions on~$d/n$ derived in~\cite{MR4065151,MR4168791} for the existence of a MLE. 
However, note that~\eqref{dncon} is only a sufficient condition and we expect \eqref{intcon} to hold more broadly.

We then proceed to study the condition number~$L/m$ explicitly. 
In Section \ref{sec:prop_regime}, we show that~$L/m$ remains bounded with high probability (w.h.p.) as~$n,d$ grow, under the proportional asymptotics regime for small enough~$r=\lim_{n,d\rightarrow\infty} d/n$, in the Gaussian design case. More specifically, in that setting, we show that there exists~$\kappa^*\in[1,\infty)$ independent of~$n,d,r$ such that~$\mathbb{P}(L/m\leq\kappa^*)\rightarrow 1$ as~$n,d\rightarrow\infty$ subject to~$\lim_{n,d\rightarrow\infty} d/n = r$. 
The value of $\kappa^*$ depends on the data-generating parameters (namely~$\beta_0$ and $\gamma_0$ defined after Assumption~\ref{as1} below) and becomes larger as 
the data become closer to separable; see Proposition~\ref{map} for a precise statement. 
Finally, in Section \ref{sec:disprop}, we consider the disproportional regime~$d/n\rightarrow0$ ($r=0$), where we give an explicit and numerically tractable expression for $\kappa^*$ (see Theorem~\ref{nmt}).
For example, in the so-called
null case (i.e.~$\beta_0=\gamma_0=0$), the upper bound reads~$\kappa^*=\pi$; see Remark \ref{imt}(\ref{imt1}).

Binary regression models are the prototypical application of non-asymptotic convergence analysis, especially in the Markov chain Monte Carlo literature \cite{dalalyan2017theoretical,Chewi26Book}. 
The resulting bounds depend non-trivially (typically linearly) on the posterior condition number, so that controlling the latter is crucial to obtain practically meaningful results.
There the standard approach is to add a ridge penalization/Gaussian prior, which automatically makes the posterior strongly log-concave, but with a condition number that scales like~$n/d$, see e.g.\ \cite[Section~6.2]{dalalyan2017theoretical},~\cite[Section~2.1]{chak2025c}.
This leads to upper bounds that diverge as $n\to\infty$, even if $d$ is fixed, which stands in contrast to classical statistical asymptotic theory~\cite[Section~10.2]{MR1652247}.
Our results instead show that probit condition numbers remain bounded in various regimes of statistical interest, especially when $n/d\to\infty$. This allows us to derive sharp and meaningful complexity results as a direct application of classical algorithmic theory for log-concave objective functions, as well as for more general (not necessarily Gaussian nor proper) prior densities~$\pi$.

Our results stand in contrast to the situation for logistic regression. In the latter case, the Hessian of the negative log-likelihood is~$\mathbb{R}^{d+1}\ni\theta\mapsto\sum_{i=1}^n \rho(\theta\cdot \bar{x_i}) \bar{x}_i\bar{x}_i^{\top}$, for some~$\rho:\mathbb{R}\rightarrow(0,\infty)$ with~$\lim_{z\rightarrow\pm \infty}\rho(z)=0$, as opposed to~\eqref{probit} and~\eqref{pro} below for the probit case. Since~$\cup_{i=1}^n\{\theta\in\mathbb{R}^{d+1}:\theta\cdot \bar{x}_i=0\}$ has Lebesgue measure zero, there are directions along which the Hessian converges to zero at infinity, which always yields an infinite condition number.
On the other hand, we note that logistic regression can exhibit \textit{locally} valid condition numbers without any positive curvature from the regularization/prior, 
which can suffice to obtain good complexities for 
sampling algorithms~\cite{chak2025c}, even if the resulting proof techniques become significantly more delicate and involved.

\paragraph{Notation}
For~$v_0\in\mathbb{R}$,~$v = (v_1,\dots,v_d)\in\mathbb{R}^d$, we use~$(v_0,v^{\top})$ to denote the row vector~$(v_0,\dots,v_d)$. Elements of~$\mathbb{R}^d$ are treated as column vectors. The notation~$\Phi,\phi$ is used for the c.d.f. and p.d.f. respectively of the standard normal distribution. 
For any~$i\in\mathbb{N}\cap[1,n]$, we denote by~$(x_{ji})_{j=1}^d$ the elements of~$x_i = (x_{1i},\dots,x_{di})$. We denote~$\mathbb{S}^d:=\{v\in\mathbb{R}^{d+1}\,:\,|v|=1\}$. 

\section{Existence of a finite condition number}\label{exs}

In this section, we give a characterization for a finite condition number, analogous to that for the existence of the MLE (see~\cite[Theorem~1]{MR1185200},~\cite[Section~2.1]{MR4168791} or~\eqref{condsq2} below). A sufficient condition on~$d/n$ is then provided, which is valid with exponentially (in-$n$) high probability for Gaussian covariates. Like~\cite{MR4065151}, this sufficient condition will depend on the signal strength. 
As noted in the introduction, the non-existence of an MLE implies that the log-likelihood is not strongly concave, in which case the condition number is infinite. Thus the regimes where the MLE does not exist, as provided in~\cite{MR4168791}, are also regimes where the condition number is infinite. We note the regimes for existence derived below do not cover the complement of the regimes derived in~\cite{MR4168791}.

We start by writing down the Hessian of~$l$ and stating some basic facts that will be used. The Hessian of~$l$ is the function~$H:\mathbb{R}^{d+1}\rightarrow\mathbb{R}^{(d+1)\times (d+1)}$ given by
\begin{equation}\label{probit}
 H(\theta) = D^2l(\theta) = \textstyle\sum_{i=1}^n \psi(y_i\bar{x}_i\cdot\theta)
\bar{x}_i\bar{x}_i^\top,
\end{equation}
where $\psi:\mathbb{R}\rightarrow\mathbb{R}$ is given by
\begin{equation}\label{psidef}
\psi(z) = \frac{\phi(z)}{\Phi(z)}\bigg(z + \frac{\phi(z)}{\Phi(z)}\bigg).
\end{equation}
The function~$\psi$ can be written as~$\psi(z)=-f'(z)$, where~$f:\mathbb{R}\rightarrow\mathbb{R}$ is given by~$f(z)=\phi(z)/\Phi(z) = \phi(-z)/\int_{-z}^{\infty}\phi(z')dz'$. Thus by~\cite{MR54890} and some standard asymptotic expansions of~$\Phi$, we have
\begin{equation}\label{pro}
\psi(0)=2/\pi, \quad\lim_{z\rightarrow\infty}\psi(z) = 0,\quad \lim_{z\rightarrow-\infty}\psi(z) = 1, \qquad \psi(z)\in(0,1),\quad \psi'(z)<0\quad\forall z\in\mathbb{R}.
\end{equation}
Our goal is to characterize/upper bound the ratio~$L/m$, where
\begin{equation}\label{eq:def_m_L}
L = \textstyle \sup_{\theta\in\R^{d+1}} \lambda_{\max}(H(\theta)) ,\qquad
m = \inf_{\theta\in\R^{d+1}} \lambda_{\min}(H(\theta))\,,
\end{equation}
with $\lambda_{\max}$ and $\lambda_{\min}$ denoting respectively the maximum and minimum eigenvalue of a symmetric positive semidefinite matrix.
Throughout, we refer to $L/m\in[1,\infty]$ as the condition number.

\subsection{Characterization}
To begin with, we do not make any probabilistic (data-generating or distributional) assumptions on~$y_i,\bar{x}_i$ (see Assumption~\ref{as1} below) in the next Theorem~\ref{aslp}, just that~$y_i\in\{-1,1\}$ and~$\bar{x}_i\in\mathbb{R}^{d+1}$ for all~$i$.

\begin{theorem}\label{aslp}
Let~$P:\mathbb{N}\cap[0,n]\rightarrow\mathbb{R}$ and~$s^*$ be given by
\begin{equation}\label{psdef}
P(s) := \inf_{v\in\mathbb{S}^d}\min_{I:|I|=s}\sum_{i\in I}|\bar{x}_i\cdot v|^2, \qquad s^*:=\inf_{u\in\mathbb{S}^d}\sum_{i=1}^n\mathds{1}_{(-\infty,0]}(y_i\bar{x}_i\cdot u)\,.
\end{equation}
The quantities~$L,m$ given by~\eqref{eq:def_m_L} satisfy the following.
\begin{enumerate}
\item \label{d0}
It holds that
\begin{equation}\label{d0eq}
L/m <\infty \iff \textrm{Span}((\bar{x}_i)_{i\in I_u})=\mathbb{R}^{d+1}\quad\forall \,u\in\mathbb{S}^d,
\end{equation}
where~$I_u = \{i\in[1,n]\cap\mathbb{N}:y_i\bar{x}_i\cdot u\leq 0\}$.
\item \label{d1}
If~$(\bar{x}_i)_{i\in I}$ spans~$\mathbb{R}^{d+1}$ for any choice of~$I\subset[1,n]\cap \mathbb{N}$ with size~$|I|> d$, then 
it holds that
\begin{equation}\label{condsq}
L/m<\infty\iff s^*> d.
\end{equation}
\item \label{d2}
It holds that
\begin{equation*}
L/m\leq (\pi/2)\lambda_{\max}(\textstyle\sum_{i=1}^n\bar{x}_i\bar{x}_i^{\top})/P(s^*).
\end{equation*}
\end{enumerate}
\end{theorem}
\begin{remark}\label{re0}
\begin{enumerate}
\item \label{re0n} The right-hand condition in~\eqref{d0eq} is, as must be the case, a strengthening of the corresponding condition for the existence of the MLE (see~\cite[Theorem~1]{MR1185200} or~\cite[Section~2.1]{MR4168791}). To see this (without Theorem~\ref{aslp}): let~$u\in\mathbb{S}^d$. The right-hand condition in~\eqref{d0eq} implies that there exists~$i\in I_u$ such that~$u\cdot \bar{x}_i\neq0$, which implies by definition of~$I_u$ that~$y_i\bar{x}_i\cdot u<0$. Moreover, it implies there exists~$i'\in I_{-u}$ such that~$u\cdot\bar{x}_{i'}\neq0$, which implies~$y_{i'}\bar{x}_{i'}\cdot u>0$.
\item The interest in the setting of Theorem~\ref{aslp}(\ref{d1}) lies in continuous data. For~$\bar{x}_i=(1,x_i^{\top})^{\top}$ with~$x_i$ independent and identically distributed (i.i.d.) with an absolutely continuous distribution, 
the assumption in Theorem~\ref{aslp}(\ref{d1}) is verified almost surely below in Proposition~\ref{assp}. 
The analogous assumption also holds without the intercept in that case, but we omit this statement.
\item \label{re2} The inequality~$s^*>d$ alone is not strictly stronger than the corresponding condition
\begin{equation}\label{condsq2}
\inf_{u\in\mathbb{S}^d}\sum_{i=1}^n\mathds{1}_{(-\infty,0)}(y_i\bar{x}_i\cdot u)
> 0
\end{equation}
for the existence of the MLE, since the indicator sets are~$(-\infty,0)$ rather than~$(-\infty,0]$. However under the assumption in Theorem~\ref{aslp}(\ref{d1}),~$s^*>d$ does imply~\eqref{condsq2} (which is seen without using Theorem~\ref{aslp}). 
\item The quantity~$s^*/n$ coincides with the Tukey depth of the origin with respect to~$\{y_i \bar{x}_i\}_{i=1}^n$, and~\eqref{condsq} states that the condition number is finite if and only if this depth is sufficiently high.
\end{enumerate}
\end{remark}

\begin{proof}
By~\eqref{pro}, for any~$v\in\mathbb{S}^d$, the Hessian~\eqref{probit} satisfies
\begin{equation}\label{sps}
\inf_{\theta\in\mathbb{R}^{d+1}}v^{\top}H(\theta)v = \inf_{\theta\in\mathbb{R}^{d+1}}\sum_{i=1}^n\psi(y_i\bar{x}_i\cdot \theta)|\bar{x}_i\cdot v|^2 \geq \inf_{\theta\in\mathbb{R}^{d+1}}\frac{2}{\pi}\sum_{i=1}^n\mathds{1}_{(-\infty,0]}(y_i\bar{x}_i\cdot \theta)|\bar{x}_i\cdot v|^2 =: S_v.
\end{equation}
We have~$S_v\geq 
(2/\pi)\inf_{\theta\in\mathbb{R}^{d+1}}
\min_I\sum_{i\in I}|\bar{x}_i\cdot v|^2$,
where the minimum is over sets of indices~$I\subset[1,n]\cap\mathbb{N}$ of size~$|I|=\sum_{i=1}^n \mathds{1}_{(-\infty,0]}(y_i\bar{x}_i\cdot \theta)$. 
Since~$|I|\mapsto\min_{J:|J|=|I|}\sum_{i\in J}|\bar{x}_i\cdot v|^2$ is an increasing function, it holds that
\begin{equation}\label{dre}
S_v 
\geq
\frac{2}{\pi}\min_I\sum_{i\in I}|\bar{x}_i\cdot v|^2,
\end{equation}
where the minimum here is over~$I$ of size~$|I|=s^*$. 
Combining~\eqref{dre} with $L\leq \lambda_{\max}(\textstyle\sum_{i=1}^n\bar{x}_i\bar{x}_i^{\top})$, which follows directly from~\eqref{pro},
yields assertion~\ref{d2}. Assertion~\ref{d1} follows from~\ref{d0}, so it remains to prove~\ref{d0}. 
Recall that, for any~$u\in\mathbb{S}^d$, if $(\bar{x}_i)_{i\in I_u}$ spans~$\mathbb{R}^{d+1}$, then one has $\inf_{v\in\mathbb{S}^d}\sum_{i\in I_u}|\bar{x}_i\cdot v|^2>0$.
Thus, since the number of possible~$I_u$ over~$u\in\mathbb{S}^d$ is finite, the right-hand condition in~\eqref{d0eq} implies~$\inf_{u,v\in\mathbb{S}^d}\sum_{i\in I_u}|\bar{x}_i\cdot v|^2>0$. 
Subsequently, by~\eqref{sps}, this 
yields~$m\geq \inf_{v\in\mathbb{S}^d}S_v>0$. 
We have proved the right-to-left implication in~\eqref{d0eq}. 

For the other direction in~\eqref{d0eq}, suppose there exist~$u^*\in\mathbb{S}^d$ 
with~$\textrm{Span}((\bar{x}_i)_{i\in I_{u^*}})\neq \mathbb{R}^{d+1}$. 
It follows that the Hessian~\eqref{probit} satisfies 
\begin{equation}\label{jq}
\inf_{c\in(0,\infty),v\in\mathbb{S}^d} \sum_{i=1}^n\psi(y_i(c\bar{x}_i\cdot u^*))|\bar{x}_i\cdot v|^2 \leq \inf_{v\in\mathbb{S}^d} \lim_{c\rightarrow\infty}\sum_{i=1}^n\psi(y_i(c\bar{x}_i\cdot u^*))|\bar{x}_i\cdot v|^2  \leq  \inf_{v\in\mathbb{S}^d}\sum_{i\in I_{u^*}}|\bar{x}_i\cdot v|^2.
\end{equation}
Since we can always choose~$v\in\mathbb{S}^d$ orthogonal to the span of~$(\bar{x}_i)_{i\in I_{u^*}}$, the right-hand side of~\eqref{jq} is zero. Therefore there can be no~$m>0$ with~$v^{\top}Hv\geq m$ for all~$v\in\mathbb{S}^d$.
\end{proof}

In the next Proposition~\ref{assp}, we verify the assumption of Theorem~\ref{aslp}(\ref{d1}) for~$\bar{x}_i = (1,x_i^{\top})^{\top}$ with~$x_i$ that are i.i.d. with an absolutely continuous distribution. The statement for~$\bar{x}_i=x_i$ follows similarly. 

\begin{proposition}\label{assp}
Assume~$(x_i)_{i=1}^n$ is an i.i.d. sequence of~$\mathbb{R}^d$-valued r.v.'s 
such that the distribution of~$x_i$ is absolutely continuous. 
It holds almost surely (a.s.) that~$\textrm{Span}(\{(1,x_i^{\top})^{\top}\}_{i\in I}) = \mathbb{R}^{d+1}$ for any set of indices~$I\subset\mathbb{N}\cap[1,n]$ with size~$|I|= d+1$. 
\end{proposition}
\begin{proof}
It suffices to show that the matrix~$M$ with its~$i^{\textrm{th}}$ row equal to~$(1,x_i^{\top})$ is a.s.\ invertible. This is true if and only if its determinant is a.s.\ nonzero. Its determinant is~$\bar{P}((x_i)_{i\in I}) := P((x_{ji})_{i\in I,j\in[1,d]\cap\mathbb{N}})$ for a nonzero polynomial function~$P$. 
Therefore by~\cite{ct},~$P$ attains zero only on a Lebesgue null set. By the assumption on~$(x_i)_{i\in I}$, the assertion follows. 
\end{proof}

\subsection{Sufficient condition between~$n$ and~$d$}\label{nds}
The goal of the remainder of this section is to prove a quantitative lower bound for~$s^*$ given by~\eqref{psdef} w.h.p.\ under a Gaussian covariate assumption. Whether or not the lower bound is greater than~$d$ will depend only on~$(d+1)/n$. 
We work under the intercept
\begin{equation}\label{intercept}
~\bar{x}_i=(1,x_i^{\top})^{\top}
\end{equation}
case, with the following assumption.
\begin{assumption}\label{as1}
The covariates~$(\bar{x}_i)_{i=1}^n$ are given by~\eqref{intercept} with~$(x_i)_{i=1}^n$ that are i.i.d. r.v.'s such that~$x_i\sim N(0,I_d)$ for all~$i$. Moreover, the responses~$(y_i)_{i=1}^n$ are conditionally independent given~$(x_i)_{i=1}^n$, and there exists~$\bar{\beta}\in\mathbb{R}^{d+1}$ such that~$\frac{y_i + 1}{2}|(x_i)_{i=1}^n\sim \textrm{Bernoulli}(\Phi(\bar{x}_i\cdot\bar{\beta}))$ for all~$i$.
\end{assumption}

The~$x_i\sim N(0,I_d)$ condition in Assumption~\ref{as1} is not restrictive, in the sense that it allows us to obtain, for the~$x_i\sim N(0,\Sigma)$ case with positive definite~$\Sigma$, results analogous to those below on~$L/m$. 
Indeed, under~$x_i\sim N(0,\Sigma)$, we may rewrite the loglikelihood as
\begin{equation*}
l(\theta)
= -\sum_{i=1}^n\log\Phi(y_i\hat{x}_i\cdot\bar{\Sigma}^{1/2}\theta), 
\textrm{where } \hat{x}_i:=\bar{\Sigma}^{-1/2}\bar{x}_i \textrm{ with } \bar{\Sigma}:=\begin{pmatrix}1 & 0 \\ 0 & \Sigma\end{pmatrix},
\end{equation*}
so that generalizing from~$I_d$ to~$\Sigma$ is equivalent to considering the preconditioned likelihood in the sense of~\cite[Section~4.2]{dalalyan2017theoretical}. 
In particular, this yields~$L/m\leq \kappa\cdot \lambda_{\max}(\bar{\Sigma}^{1/2})^2/\lambda_{\min}(\bar{\Sigma}^{1/2})^2 = \kappa\cdot \max(1,\lambda_{\max}(\Sigma))/\min(1,\lambda_{\min}(\Sigma))$, where~$\kappa$ is any of the bounds on~$L/m$ in the~$x_i\sim N(0,I_d)$ case, as obtained below.

Let~$\gamma_0\geq 0$ be such that~$\gamma_0^2=\textrm{Var}(x_i\cdot \beta)$ (for any~$i$). Under Assumption~\ref{as1}, by rotational invariance on the regression parameter space, we may assume  w.l.o.g.\ that~$\beta = (\gamma_0,0,\dots,0)$. 
As in the introduction, under Assumption~\ref{as1}, we denote by~$\beta_0\in\mathbb{R}$ and~$\beta\in\mathbb{R}^d$ the elements of~$\bar{\beta}=(\beta_0,\beta^{\top})^{\top}$. 
Throughout, we assume that $\beta_0$ and $\gamma_0$ are independent of $n$ and $d$.

\begin{theorem}\label{mth}
Under Assumption~\ref{as1} and~$n> d$, for any~$\epsilon>0$, it holds with probability at least~$1-\exp(-2\epsilon^2n)$ that 
\begin{equation}\label{qrh}
s^* 
\geq n\bigg(h- \epsilon - \bigg(\frac{2\log(en/(d+1))}{n/(d+1)}\bigg)^{\!1/2}\,\bigg),
\end{equation}
where~$s^*$ is given by~\eqref{psdef} and, with~$Z\sim N(0,1)$,~$h>0$ is given by
\begin{equation}\label{hdef}
h=\mathbb{E}[\Phi(-|\beta_0 + \gamma_0 Z|)].
\end{equation}
\end{theorem}
\begin{remark}\label{imr}
\begin{enumerate}
\item 
In the null case~$\beta_0=\gamma_0=0$, we have~$h=1/2$. By definition of~$\Phi$, we have~$h\leq 1/2$ for all~$\beta_0,\gamma_0$. Note that~$h$ coincides with the Bayes risk~$\inf_c\mathbb{P}(c(\bar{x})\neq y)$, where the infimum is over all measurable~$c:\mathbb{R}^{d+1}\rightarrow\{-1,1\}$, under Assumption~\ref{as1}.
\item \label{imr2} We note that the infimum in the definition~\eqref{psdef} of~$s^*$ is measurable w.r.t. the probability space, because we may write~$\mathds{1}_{(-\infty,0]} = 1-\mathds{1}_{(0,\infty)}$, which leads to an infimum over (all) rational vectors.
\item With the notation of the paragraph following Assumption~\ref{as1}, $\bar{x}_i\cdot u= \hat{x}_i\cdot\bar{\Sigma}^{1/2}u$, where~$\bar{\Sigma}^{1/2}$ is invertible. The sum in~\eqref{psdef} is unchanged by positive rescaling of a direction, so~$s^*$ is the same for~$(\bar{x}_i)_i$ and~$(\hat{x}_i)_i$. Since~$(\hat{x}_i,y_i)_i$ satisfy Assumption~\ref{as1} with the same~$\beta_0$ and~$\gamma_0$, Theorem~\ref{mth} holds for~$x_i\sim N(0,\Sigma)$ with positive definite~$\Sigma$.
\end{enumerate}
\end{remark}
\begin{proof}
Let~$\mathcal{G}$ be the set of functions acting on~$\mathbb{R}^{d+1}$ given by
\begin{equation}\label{halfs}
\mathcal{G} = \{ \mathbb{R}^{d+1} \ni z\mapsto \mathds{1}_{(-\infty,0]}(z\cdot u) : u\in\mathbb{S}^d\}
\end{equation}
and let~$\mathcal{G}' = \{2g-1:g\in\mathcal{G}\}$. We denote by~$\mathcal{R}_n(\mathcal{G}), \mathcal{R}_n(\mathcal{G'})$ the Rademacher complexities~\cite[Definition~3.2]{MR3931734} of~$\mathcal{G},\mathcal{G}'$ respectively with respect to the distribution of~$y_1\bar{x}_1$ in $\mathbb{R}^{d+1}$. By definition, we have~$\mathcal{R}_n(\mathcal{G}) = \mathcal{R}_n(\mathcal{G}')/2$. 
By Theorem~3.3 in~\cite{MR3931734}, for any~$\delta>0$, it holds with probability at least~$1-\delta$ that 
\begin{equation}\label{jao}
n^{-1}s^*
\geq \inf_{u\in\mathbb{S}^d}\mathbb{P}(y_1\bar{x}_1\cdot u\leq 0) -2\mathcal{R}_n(\mathcal{G}) - \bigg(\frac{\log(1/\delta)}{2n}\bigg)^{1/2}.
\end{equation}
By Corollary~3.8 and Corollary~3.18 both again in~\cite{MR3931734}, and using that the VC-dimension of the class of homogeneous halfspaces in~$\mathbb{R}^{d+1}$ is~$d+1$~\cite[Theorem~9.2]{Shalev-Shwartz_Ben-David_2014} 
we have
\begin{equation}\label{jak}
\mathcal{R}_n(\mathcal{G}) \leq \frac{1}{2}\mathcal{R}_n(\mathcal{G}') \leq \frac{1}{2}\bigg(\frac{2(d+1)\log(en/(d+1))}{n}\bigg)^{1/2}.
\end{equation}
Moreover, using that the signal comes only from the first coordinate (namely~$\beta=(\gamma_0,0,\dots,0)$) and that $\bar{x}_1=(1,x_1^{\top})^{\top}$, it holds for any~$u=(u_0,\dots,u_d)\in\mathbb{S}^d$ that
\begin{align*}
&\mathbb{P}(y_1\bar{x}_1\cdot u\leq 0)\\
&\quad= 
\mathbb{E}[\mathds{1}_{\{\bar{x}_1\cdot u\leq 0\}}\mathds{1}_{\{y_1=1\}}] 
+ 
\mathbb{E}[\mathds{1}_{\{\bar{x}_1\cdot u\geq 0\}}\mathds{1}_{\{y_1=-1\}}]\\
&\quad=\mathbb{E}[\mathds{1}_{\{u_0 + u_1x_{11} + |(u_2,\dots,u_d)|Z\leq 0 \}}\Phi(\beta_0 + \gamma_0x_{11})] + \mathbb{E}[\mathds{1}_{\{u_0 + u_1x_{11} + |(u_2,\dots,u_d)|Z\geq 0 \}}\Phi(-\beta_0 - \gamma_0x_{11})],
\end{align*}
where~$Z\sim N(0,1)$ is independent of~$x_{11},y_1$. This implies
\begin{align*}
\mathbb{P}(y_1\bar{x}_1\cdot u\leq 0)
&=
\mathbb{E}[\Phi(-(u_0 + u_1x_{11})/|(u_2,\dots,u_d)|)\Phi(\beta_0 + \gamma_0x_{11})]\\
&\quad+ \mathbb{E}[\Phi((u_0 + u_1x_{11})/|(u_2,\dots,u_d)|)\Phi(-\beta_0 - \gamma_0x_{11})],
\end{align*}
where~$\Phi(+\infty):=1$ and~$\Phi(-\infty):=0$. 
Further, we have
\begin{align}
\mathbb{P}(y_1\bar{x}_1\cdot u\leq 0)\nonumber
&=\mathbb{E}[(1-\Phi((u_0 + u_1x_{11})/|(u_2,\dots,u_d)|))\Phi(\beta_0 + \gamma_0x_{11})]\nonumber\\
&\quad+ \mathbb{E}[\Phi((u_0 + u_1x_{11})/|(u_2,\dots,u_d)|)(1-\Phi(\beta_0 + \gamma_0x_{11}))],\nonumber\\
&=\mathbb{E}[\Phi(\beta_0 + \gamma_0x_{11})+ \Phi((u_0 + u_1x_{11})/|(u_2,\dots,u_d)|)(1-2\Phi(\beta_0 + \gamma_0x_{11}))]\nonumber\\
&\geq \mathbb{E}[\Phi(\beta_0 + \gamma_0x_{11})+\mathds{1}_{[0,\infty)}(\beta_0 + \gamma_0x_{11})(1-2\Phi(\beta_0 + \gamma_0x_{11}))]\nonumber\\
&= \mathbb{E}[\Phi(-|\beta_0 + \gamma_0x_{11}|)],\label{thg}
\end{align}
where we note that the infima over~$u$ of the left and right-hand sides of~\eqref{thg} are in fact equal, but we omit the proof. 
Substituting this and~\eqref{jak} into~\eqref{jao}, then taking~$\delta = e^{-2\epsilon^2 n}$ concludes the proof.
\end{proof}
In~\eqref{qrh}, near~$\epsilon=0$, there always exists~$s\in(0,1)$ such that~$(d+1)/n\leq s$ implies that the right-hand side of~\eqref{qrh} is greater than~$d$. Together with Theorem~\ref{aslp} and Proposition~\ref{assp}, we obtain a condition on~$(d+1)/n$ under which the condition number is finite w.h.p., as formalized in the following Corollary~\ref{cor1}. 
\begin{corollary}\label{cor1}
Under Assumption~\ref{as1}, 
there exists~$r\in(0,1)$ such that for any~$\bar{r}\in(0,r)$, if~$(d+1)/n\leq \bar{r}$, then the condition number is finite with probability at least~$1-\exp(-cn)$ for some constant~$c>0$ independent of~$d,n$. Moreover,~$r$ is the solution to
\begin{equation}\label{maco}
r+(2r\log(e/r))^{1/2} = h,
\end{equation}
where~$h$ is given by~\eqref{hdef} with~$Z\sim N(0,1)$.
\end{corollary}
The condition~\eqref{maco} is numerically tractable given~$\beta_0,\gamma_0$, in the same way as in~\cite{MR4065151}, but we don't claim that~$r$ dictated by~\eqref{maco} is the optimal such~$r$ (which is the case in~\cite{MR4065151}).

\section{Asymptotically constant condition number}\label{sec:prop_regime}

In this section, we show under Assumption~\ref{as1} that the condition number is upper bounded w.h.p. as~$n,d\rightarrow\infty$ for~$r=\lim_{n,d\rightarrow\infty}d/n$ small enough, by a constant independent of~$n,d,r$. 
Of course, a non-trivial upper bound is only possible when a finite condition number exists. 
Indeed, our Proposition~\ref{map} will necessitate a stronger condition~\eqref{sal}, compared to that obtained with~\eqref{condsq} and~\eqref{qrh} above for the existence of a finite condition number. Throughout, we assume that the probability space is complete. 

First we state a slight modification of a result from~\cite{MR3612870}. In the following, we refer to~\cite[Definition~2.2]{MR3612870} for a definition of VC-dimension. 

\begin{lemma}[A variant of Lemma~2.3 in~\cite{MR3612870}]\label{lmlem}
Let $X_1,\dots,X_n$ be i.i.d. copies of an~$\mathbb{R}^{d+1}$-valued random variable~$X$, and let~$\mathcal{F}$ be a class of measurable real-valued functions. Fix~$c\in (0,1)$,~$\beta'\in(0,1]$, and~$u \geq 0$, and suppose that
\begin{equation}\label{gea}
\inf_{f\in\mathcal{F}} \mathbb{P}(|f(X)|>u)\geq \beta'.
\end{equation}
Set $\mathcal{G}_u=\{\mathds{1}_{\{|f|>u\}}:f\in\mathcal{F}\}$, and assume
that $\mathcal{G}_u$ 
has VC-dimension at most $d'\in\mathbb{N}$. 
Assume moreover that~$
\sup_{g\in\mathcal{G}_u}|n^{-1}\sum_{i=1}^n g(X_i)-\mathbb{E}[g(X)]|$ and~$
\inf_{g\in\mathcal{G}_u}\sum_{i=1}^n g(X_i)$ are measurable (w.r.t. the probability space). 
Then there exist absolute constants $C_0,c_0>0$ such that, if $n\geq C_0d'/(c\beta')^2$, then, with probability at least $1-\exp(-c_0(c\beta')^2n)$,
\[\inf_{f\in\mathcal{F}}\big|\{1\leq i\leq n:|f(X_i)|>u\}\big|
\geq(1-c)\beta'n.\]

\end{lemma}
\begin{proof}
We follow the proof of \cite[Lemma~2.3]{MR3612870}, adopting its notation and denoting by $c_1,c_2>0$ the absolute constants appearing there. Only two modifications are needed. First, in place of~(2.3) therein, we require $\mathbb{E}[G(X_1,\dots,X_n)]\leq c_2\sqrt{d'/n}\leq c \beta'/2$ which holds under the assumption $n \geq C_0d'/(c\beta')^2$ upon setting $C_0:=(2c_2)^2$. 
Second, in place of~$t=n(\beta')^2/(16c_1^2)$, we take~$t=n(c\beta')^2/(2c_1)^2$. The resulting deviation term is $c\beta'/2$, and hence, outside an
event of probability at most $e^{-t} =\exp\left\{-n(c\beta')^2/(2c_1)^2\right\}, $ we have $G(X_1,\dots,X_n)\leq c\beta'$. Consequently, simultaneously
for every $f\in\mathcal{F}$,
$n^{-1}\left|\left\{1\leq i\leq n:|f(X_i)|>u\right\}
\right|\geq \mathbb{P}(|f(X)|>u)-c\beta'\geq(1-c)\beta',$
where the last inequality follows from~\eqref{gea}. The assertion follows with $c_0=(2c_1)^{-2}$.
\end{proof}

\begin{proposition}\label{map}
Let Assumption~\ref{as1} hold.
Let~$h$ be given by~\eqref{hdef}. Then there exist constants $C_h\in(0,1)$ and $\kappa^*<\infty$, depending only on $h$, such that, for every fixed $r\in[0,C_h]$,
\begin{equation*}
\mathbb{P}(L/m> \kappa^*)
\rightarrow 0
\end{equation*}
as~$n,d\rightarrow \infty$ with~$d/n \rightarrow r$. 
\end{proposition}

\begin{proof}
Let~$\mathcal{F}$ be the set of functions acting on~$\mathbb{R}^{d+1}$ given by
\begin{equation}\label{norms}
\mathcal{F} = \{ \mathbb{R}^{d+1} \ni x\mapsto |v\cdot x|^2 : v\in\mathbb{S}^d\}.
\end{equation}
Let~$\epsilon'=h/12$; recall from \eqref{hdef} that $h\in(0,1/2]$. For
$v=(v_0,\widetilde v)\in\mathbb{S}^d$, rotational invariance gives $v\cdot\bar{x}_1 \stackrel{d}{=} v_0+(1-v_0^2)^{1/2}Z,$ with $Z\sim N(0,1).$ The map~$F(a,t) := \mathbb{P}( |a+(1-a^2)^{1/2}Z|\leq t)$ is continuous on~$[-1,1]\times[0,1/2]$ and satisfies $F(a,0)=0$. Compactness therefore gives~$\sup_{a\in[-1,1]}F(a,t)\rightarrow 0 $ as $t \downarrow 0$. Then, there exists $t_h\in(0,1/2]$, depending only on~$h$ such that~$\sup_aF(a,t_h)\leq\epsilon'$. Setting $\delta':=t_h^2$, we obtain, for every $d$, 
\begin{equation}\label{eq:uniform-small-ball}
\inf_{v\in\mathbb{S}^d} \mathbb{P}\left(|v\cdot\bar{x}_1|^2 >\delta'\right) \geq 1-\epsilon',
\end{equation}
which is~\eqref{gea} with~$\mathcal{F}$ given by~\eqref{norms},~$X_i=\bar{x}_i = (1,x_i^{\top})^{\top}$,~$u=\delta'$ and~$\beta'=1-\epsilon'$. 

We define~$C_h'>0$ to be a constant such that if~$r\in(0,C_h']$, then~$(2r\log(e/r))^{1/2}\leq h/2$.
Let~$r\in[0,C_h']$. 
Let~$c^*\in[h/6,h/3]\subset (0,1/6]$ be given by
\begin{equation}\label{csdef}
c^*=\begin{cases}
(h-(2r\log(e/r))^{1/2})/3 &\textrm{if }r\neq 0\\
h/3 &\textrm{if }r= 0.
\end{cases}
\end{equation}
Since~$\epsilon'=h/12$ satisfies~$(\frac{1}{1-\epsilon'}-1)(\frac{3}{2(1-\epsilon')}-1)^{-1} = 2\epsilon'/(1+2\epsilon')< 2\epsilon' \leq c^*$, we have
\begin{equation}\label{cr2}
c^*<1-\frac{1}{1-\epsilon'} + \frac{3c^*}{2(1-\epsilon')}.
\end{equation}

With these quantities and having obtained~\eqref{eq:uniform-small-ball}, we apply Lemma~\ref{lmlem} with~$u=\delta'
$,~$\beta'=1-\epsilon'$,~$c=c^*$, where we use Lemma~6.1 in~\cite{chak2025c} for the requirement on the VC-dimension. Note that~$\mathcal{F}$ as in~\eqref{norms} is not the same as the corresponding set in~\cite[Lemma~6.1]{chak2025c}, but it is a smaller set, so the VC-dimension is controlled by the same estimate. Note also that the measurability of~$
\sup_{v\in\mathbb{S}^d}|n^{-1}\sum_{i=1}^n \mathds{1}_{\{|X_i\cdot v|^2>u\}} - \mathbb{P}(|X\cdot v|^2>u)|$ and~$
\inf_{v\in\mathbb{S}^d}\sum_{i=1}^n \mathds{1}_{\{|X_i\cdot v|^2>u\}}$ (as required in Lemma~\ref{lmlem}) hold by the following. Let~$f^{\dagger}:\Omega\times \mathbb{S}^d\rightarrow\mathbb{R}$ be given by~$f^{\dagger}(\omega,v)=|n^{-1}\sum_{i=1}^n \mathds{1}_{\{|X_i\cdot v|^2>u\}} - \mathbb{P}(|X\cdot v|^2>u)|$. The function~$f^{\dagger}$ is measurable, and so the set~$\{f^{\dagger}>a\}$ is measurable for any~$a\in\mathbb{R}$. By a standard projection measurability result (Proposition~8.4.4 in~\cite{MR3098996}), we have that~$\{\sup_{v\in\mathbb{S}^d}f^{\dagger}(\cdot,v)>a\}$, which is the projection of~$\{f^{\dagger}>a\}$ onto~$\Omega$, is universally measurable as required (by completeness of the probability space). Similar arguments apply for~$\inf_{v\in\mathbb{S}^d}\sum_i\mathds{1}_{\{|X_i\cdot v|^2>u\}}$. 

This application of Lemma~\ref{lmlem} yields that there exist absolute constants~$c_1,c_2>0$ such that if
\begin{equation}\label{saj}
n\geq c_1(d+1)/(c^*(1-\epsilon'))^2
\end{equation}
holds, then it holds with probability at least~$1-\exp(-c_2(c^*)^2n)$ that 
\begin{align}
\inf_{v\in\mathbb{S}^d}|\{i\in\mathbb{N}\cap[1,n]:|v\cdot \bar{x}_i|^2>
\delta'\}| &= \inf_{f\in\mathcal{F}}|\{i\in\mathbb{N}\cap[1,n]:|f(\bar{x}_i)|>
\delta'\}|\nonumber\\
&\geq (1-c^*)(1-\epsilon')n. \label{lec}
\end{align}
Note that for~$r\neq0$, there exists an absolute constant~$C>0$ such that for large enough~$n,d$ with~$\lim_{n,d\to\infty} d/n=r$, the condition
\begin{equation}\label{sal}
h-(2r\log(e/r))^{1/2}>Cr^{1/2}
\end{equation}
is sufficient for~\eqref{saj}, because~\eqref{sal} and~\eqref{csdef} imply~$((d+1)/n)^{1/2}\leq 2r^{1/2}< 6c^*/C$ for~$n,d\gg1$ with~$\lim_{n,d\to\infty} d/n=r\neq 0$, which implies~$n\geq C^2(d+1)/(6c^*)^2\geq C^2(d+1)/(12c^*(1-\epsilon'))^2$. 
Moreover, note that by the monotonicity of~$[0,1]\ni r\mapsto r\log(e/r)$, there exists~$C_h''\in(0,1)$ such that~$r\in(0,C_h'']$ implies~\eqref{sal}. 
For~$r=0$,~\eqref{saj} is automatically valid in this asymptotic regime. We define~$C_h$ (in the assertion) to be~$C_h=\min(C_h',C_h'')$.

Inequality~\eqref{lec} with~\eqref{cr2} implies 
\begin{equation}\label{lec2}
\inf_{v\in\mathbb{S}^d}|\{i\in\mathbb{N}\cap[1,n]:|v\cdot \bar{x}_i|^2>
\delta'\}| \geq n-3nc^*/2.
\end{equation}
Moreover, Theorem~\ref{mth} with~$\epsilon=
3c^*/4$ implies for~$s^*$ given by~\eqref{psdef} and large enough~$n,d$ that it holds with probability at least~$1-\exp(-2\epsilon^2n)$ that
\begin{equation}\label{ndp}
s^*\geq 2nc^*.
\end{equation}
Combining with the event that~\eqref{lec2} holds, for large enough~$n,d$ with~$\lim d/n=r\in[0,C_h]$, 
it holds with probability at least~$1-\exp(-2\epsilon^2n) - \exp(-c_2(c^*)^2n)$ that
\begin{equation}\label{psbo}
P(s^*)\geq P(\lfloor2nc^*\rfloor) \geq \delta'\cdot 3nc^*/7 \geq \delta'\cdot (n/7)\cdot (h/2),
\end{equation}
where~$P$ is given by~\eqref{psdef}.

On the other hand, since for any~$v=(v_0,v_1,\dots,v_d)\in\mathbb{S}^d$ and denoting~$\bar{v}=(v_1,\dots,v_d)$ we have
\begin{equation*}
v^{\top}\bigg(n^{-1}\sum_{i=1}^n \bar{x}_i\bar{x}_i^{\top}\bigg)v = n^{-1}\sum_{i=1}^n(v_0 + \bar{v}\cdot x_i)^2 
\leq 2v_0^2 + 2|\bar{v}|^2\cdot \lambda_{\max}\bigg(n^{-1}\sum_{i=1}^n x_ix_i^{\top}\bigg),
\end{equation*}
by Theorem~6.1 in~\cite{MR3967104}, it holds for any~$\epsilon''>0$ that
\begin{equation*}
\mathbb{P}\bigg(\lambda_{\max}\bigg(n^{-1}\sum_{i=1}^n (1,x_i^{\top})^{\top}(1,x_i^{\top})\bigg) \geq 2((1+\epsilon'') + (d/n)^{1/2})^2\bigg)\leq \exp(-n(\epsilon'')^2/2).
\end{equation*}
Therefore together with~$2((1+\epsilon'') + (d/n)^{1/2})^2\leq 2(2+\epsilon'')^2$ and~\eqref{psbo}, setting (e.g.)~$\epsilon''=1$, by Theorem~\ref{aslp}(\ref{d2}), 
the proof concludes with~$\kappa^* = \pi(2+\epsilon'')^2/(\delta'\cdot(1/7)\cdot (h/2)) = 126\pi/(\delta'h)$.
\end{proof}

\section{Condition number under disproportional asymptotics}\label{sec:disprop} 
In this section, we study the~$d/n\rightarrow 0$ regime under Assumption~\ref{as1}, and provide a numerically tractable high-probability upper bound for the condition number as~$n,d\rightarrow\infty$. This upper bound is illustrated in Figure~\ref{fig2}.

In view of the deterministic lower bound \eqref{sps} for the
Hessian \eqref{probit}, the key step is to establish, using tools from \cite{MR4628026}, a uniform law of large numbers for the class
\[\left\{ \mathbb{R}^{d+1}\ni z\mapsto \mathds{1}_{(-\infty,0]}(z\cdot\theta)(z\cdot v)^2:
\theta\in\mathbb{R}^{d+1},\ v\in\mathbb{S}^d \right\}.\]
Lemma \ref{lem:aux_pe} establishes the required convergence when $v$ ranges over a deterministic finite set, and Lemma \ref{da2} extends it to all $v\in\mathbb{S}^d$ by an $\varepsilon$-net argument.
Lemma \ref{da1} identifies the corresponding population infimum. Combining these ingredients yields Theorem \ref{nmt}.

In this section, we prove results as almost sure convergences. To do that, we implicitly consider the independent coupling of the random variables $(x_i)_{i=1}^n$ and $(y_i)_{i=1}^n$ across $n\in\{1,2,\dots\}$.

\begin{lemma}\label{lem:aux_pe} 
For each $n$, let $N_n \subset \mathbb{S}^{d}$ be deterministic with $|N_n|\leq 9^{d+1}$.
Then, under Assumption~\ref{as1}, it holds almost surely that
\begin{equation}\label{eq:aux_pe}
\max_{v \in N_n}\sup_{\theta\in\mathbb{R}^{d+1}}\bigg|\frac{1}{n}\sum_{i=1}^n \mathds{1}_{(-\infty,0]}(y_i\bar{x}_i\cdot \theta) |\bar{x}_i\cdot v|^2 - \mathbb{E}[\mathds{1}_{(-\infty,0]}(y_1\bar{x}_1\cdot \theta)|\bar{x}_1\cdot v|^2]\bigg| \longrightarrow 0
\end{equation}
as~$n,d\rightarrow\infty$ subject to~$\lim_{n,d\rightarrow\infty}d/n=0$. 
\end{lemma}
\begin{remark}[Measurability]\label{rem:measurability}
 Under Assumption~\ref{as1}, 
the standard separability argument for halfspace classes implies that all suprema in 
Lemma~\ref{da2} and its proof may almost surely be restricted to $\theta \in \mathbb{Q}^{d+1}$ and to a fixed countable dense subset of $\mathbb{S}^d$, the population terms being continuous away from the origin and all quantities continuous in $v$. Hence these suprema are measurable on the completed probability space, and empirical process inequalities for countable classes apply; see~\cite[Section~2.3]{MR4628026}.   
\end{remark}

\begin{proof}
Write $Z_i=y_i\bar x_i$, and let $Z$ be an independent copy of $Z_i$.
For $v=(v_0,\widetilde v)\in\mathbb{S}^d$, the identity $y^2=1$ and Assumption~\ref{as1} 
give $(Z\cdot v)^2=(\bar{x}\cdot v)^2 \stackrel{d}{=} \bigl(v_0+|\widetilde v|g\bigr)^2\leq 2(1+g^2)$, with~$g\sim N(0,1)$.
Consequently,
\begin{equation}\label{34p}
\mathbb{E}\big[(Z\cdot v)^4\big]=v_0^4+6v_0^2|\widetilde v|^2+3|\widetilde v|^4
=3-2v_0^4\leq 3, 
\end{equation}
and there exist absolute constants~$C_0,c_0>0$ such that
\begin{equation}\label{eq:gauss_bounds}
\quad \sup_{v\in\mathbb{S}^d}\mathbb{P}((Z\cdot v)^2>t)\leq C_0e^{-c_0t},
\quad t\geq 0\,.
\end{equation}
In particular, the variables $(Z\cdot v)^2$, $v\in\mathbb{S}^d$,
are uniformly sub-exponential; see, for example,
\cite[Chapter~2]{Vershynin_2026} or \cite[\S 2.7--2.8]{BLM}.
Define 
\[ s_n:=(d+1)\log9+2\log n, \quad M_n:=\sqrt{\frac{n}{s_n}}, \quad \alpha_n := \sqrt{\frac{d+1}{n}\log\frac{en}{d+1}}.\]
The assumption $d/n \to 0$ gives
\begin{equation}\label{eq:rates} 
\frac{s_n}{n} \to 0, \quad M_n\to\infty, \quad \frac{M_n s_n}{n}
= \sqrt{\frac{s_n}{n}} \to 0, \quad \alpha_n \to 0.\end{equation}

For $v\in N_n$, define $a_v(z):=(z\cdot v)^2\wedge M_n$ and $t_v(z):=\big((z\cdot v)^2-M_n\big)_+$. Since $(z\cdot v)^2=a_v(z)+t_v(z)$, the expression in~\eqref{eq:aux_pe} is bounded by $A_n+B_n$, where
\begin{equation*}
A_n:=\max_{v\in N_n}\sup_{\theta\in\mathbb{R}^{d+1}}\bigg|\frac{1}{n}\sum_{i=1}^n\mathds{1}_{(-\infty,0]}(Z_i\cdot\theta)\,a_v(Z_i)-\mathbb{E}\big[\mathds{1}_{(-\infty,0]}(Z\cdot\theta)\,a_v(Z)\big]\bigg|
\end{equation*}
and $B_n$ is defined analogously with $t_v$ in place of $a_v$.

We first control $B_n$. The pointwise domination $0\leq\mathds{1}_{(-\infty,0]}(z\cdot\theta)\,t_v(z)\leq t_v(z)$, holding for every $\theta$, yields
\begin{align*}
B_n&\leq\max_{v\in N_n}\frac{1}{n}\sum_{i=1}^n t_v(Z_i)+\max_{v\in N_n}\mathbb{E}\,t_v(Z)
\leq
\max_{v\in N_n}\frac{1}{n}\sum_{i=1}^n (t_v(Z_i)-\mathbb{E}\,t_v(Z))+2\max_{v\in N_n}\mathbb{E}\,t_v(Z)
.
\end{align*}
The exponential tail bound in~\eqref{eq:gauss_bounds} gives, uniformly in $v$,
\begin{equation*}
\mathbb{E}\,t_v(Z)=\int_{M_n}^{\infty}\mathbb{P}\big((Z\cdot v)^2>t\big)\,dt\leq\frac{C_0}{c_0}\,e^{-c_0M_n}.
\end{equation*}
Fix $\varepsilon\in(0,1]$. By~\eqref{eq:rates}, for all sufficiently large $n$ we have $(C_0/c_0)e^{-c_0M_n}\leq\varepsilon/2$, and hence~$2\max_{v\in N_n}\mathbb{E}\,t_v(Z)\leq\varepsilon$. 
Since $0\leq t_v(z)\leq(z\cdot v)^2$, the centered variables $t_v(Z_i)-\mathbb{E}\,t_v(Z)$ are sub-exponential uniformly in $v$, $n$ and $d$, so Bernstein's inequality (see, e.g., \cite[\S2.8]{BLM}) and a union bound over $v\in N_n$ give, for an absolute constant $c_1>0$,
\begin{equation*}
\mathbb{P}\big(B_n>2\varepsilon\big)\leq2\cdot9^{d+1}e^{-c_1\varepsilon^2n}.
\end{equation*}
The right-hand side is summable over $n$ because $d/n\to0$. Thus, the Borel--Cantelli lemma, applied for each $\varepsilon=1/k$, $k\in\mathbb{N}$, gives $B_n\to0$ almost surely as $n\to\infty$.

It remains to control $A_n$. Fix $v\in N_n$ and set
\begin{equation*}
X_v:=\sup_{\theta\in\mathbb{R}^{d+1}}\bigg|\frac{1}{n}\sum_{i=1}^n\mathds{1}_{(-\infty,0]}(Z_i\cdot\theta)\,a_v(Z_i)-\mathbb{E}\big[\mathds{1}_{(-\infty,0]}(Z\cdot\theta)\,a_v(Z)\big]\bigg|,
\end{equation*}
so that $A_n=\max_{v\in N_n}X_v$. By the symmetrization inequality \cite[Lemma~2.3.1]{MR4628026},
\begin{equation}\label{eq:symm}
\mathbb{E}X_v\leq2\,\mathbb{E}\,\mathbb{E}_{\xi}\sup_{\theta\in\mathbb{R}^{d+1}}\bigg|\frac{1}{n}\sum_{i=1}^n\xi_i\,\mathds{1}_{(-\infty,0]}(Z_i\cdot\theta)\,a_v(Z_i)\bigg|,
\end{equation}
where $(\xi_i)_{i=1}^n$ are i.i.d.\ Rademacher variables independent of $(Z_i)_i$ and $\mathbb{E}_{\xi}$ denotes the expectation with respect to them. Conditionally on $Z_1,\dots,Z_n$, the supremum in~\eqref{eq:symm} depends only on the patterns $\big(\mathds{1}_{(-\infty,0]}(Z_i\cdot\theta)\big)_{i=1}^n$, $\theta\in\mathbb{R}^{d+1}$. The class of homogeneous halfspaces in $\mathbb{R}^{d+1}$ has VC dimension $d+1$; see \cite[Theorem~9.2]{Shalev-Shwartz_Ben-David_2014}. Hence, by the Sauer--Shelah lemma \cite{Sauer1972,Shelah1972}, the number $K_n$ of such patterns satisfies, for all sufficiently large $n$, $K_n\leq (en/(d+1))^{d+1}.$
For each pattern $S\subset\{1,\dots,n\}$, the sum $n^{-1}\sum_{i\in S}\xi_i\,a_v(Z_i)$ is, conditionally on $(Z_i)_i$, centered sub-Gaussian with variance proxy at most $n^{-2}\sum_{i=1}^na_v(Z_i)^2\leq n^{-2}\sum_{i=1}^n(Z_i\cdot v)^4$, by Hoeffding's lemma \cite[\S2.6]{BLM}. The maximal inequality for finite families of sub-Gaussian variables \cite[\S2.5]{BLM}, applied to the $K_n$ patterns and their negatives, therefore bounds the right-hand side of~\eqref{eq:symm} by
\[
2\,\mathbb{E}\bigg(\frac{2\log(2K_n)}{n^2}\sum_{i=1}^n(Z_i\cdot v)^4\bigg)^{1/2} \leq 2\,\bigg(\frac{2\log(2K_n)}{n}\,\mathbb{E}\big[(Z\cdot v)^4\big]\bigg)^{1/2}
\leq C_1\alpha_n,
\]
 where $C_1>0$ is an absolute constant, the first inequality is Jensen's and the second follows from~\eqref{34p} together with $\log(2K_n)\leq 2(d+1)\log(en/(d+1))$. Hence
\begin{equation}\label{eq:mean_bound}
\sup_{v\in N_n}\mathbb{E}X_v \leq C_1\alpha_n.
\end{equation}
For fixed $v\in N_n$, the functions
\[z\mapsto\pm\left(\mathds{1}_{(-\infty,0]}(z\cdot\theta)\,a_v(z)-\mathbb{E}\big[\mathds{1}_{(-\infty,0]}(Z\cdot\theta)\,a_v(Z)\big]\right),
\quad\theta\in\mathbb{R}^{d+1},\] are centered, bounded in absolute value by $M_n$ (as $\mathds{1}_{(-\infty,0]}(z\cdot\theta)\,a_v(z)$ and its expectation both lie in $[0,M_n]$) and, by~\eqref{34p}, of variance at most $\mathbb{E}[a_v(Z)^2]\leq\mathbb{E}[(Z\cdot v)^4]\leq3$; moreover, the supremum of their sample means $n^{-1}\sum_{i=1}^n$ over $\theta$ and the two signs is precisely~$X_v$. Consequently, Bousquet's inequality \cite[Theorem~12.5]{BLM}, together with~\eqref{eq:mean_bound}, gives, for every~$s>0$,
\begin{equation}\label{eq:aux_bousquet}
\mathbb{P}\left(X_v>C_1\alpha_n + \left( \frac{2s\bigl(3+2C_1M_n\alpha_n\bigr)}{n}\right)^{1/2} + \frac{M_ns}{3n}\right)\leq e^{-s}.
\end{equation}
Taking $s=s_n$ and applying a union bound over  $v\in N_n$, we obtain

\[\mathbb{P}\left(A_n > C_1\alpha_n + \left( \frac{2s_n\left(3+2C_1M_n\alpha_n\right)}{n}
\right)^{1/2} + \frac{M_ns_n}{3n} \right) \leq 9^{d+1}e^{-s_n} = n^{-2}. \]
The deterministic threshold in this probability converges to zero
by~\eqref{eq:rates}. As $\sum_n n^{-2}<\infty$, the Borel-Cantelli lemma gives $A_n \to 0$ almost surely. 
\end{proof}

\begin{lemma}\label{da2}
Under Assumption~\ref{as1}, it holds almost surely that
\begin{equation}\label{l3eq}
\sup_{\theta\in\mathbb{R}^{d+1},v\in\mathbb{S}^d}\bigg|\frac{1}{n}\sum_{i=1}^n \mathds{1}_{(-\infty,0]}(y_i\bar{x}_i\cdot \theta) |\bar{x}_i\cdot v|^2 - \mathbb{E}[\mathds{1}_{(-\infty,0]}(y_1\bar{x}_1\cdot \theta)|\bar{x}_1\cdot v|^2]\bigg| \longrightarrow 0
\end{equation}
as~$n,d\rightarrow\infty$ subject to~$\lim_{n,d\rightarrow\infty}d/n=0$. 
\end{lemma}

\begin{proof}
For $\theta\in\R^{d+1}$, define the symmetric matrix 
\[A_\theta:=\frac{1}{n}\sum_{i=1}^n\mathds{1}_{(-\infty,0]}(y_i\bar{x}_i\cdot\theta)\,\bar{x}_i\bar{x}_i^{\top}-\mathbb{E}\big[\mathds{1}_{(-\infty,0]}(y_1\bar{x}_1\cdot\theta)\,\bar{x}_1\bar{x}_1^{\top}\big],\] so that the expression inside the absolute value in \eqref{l3eq} equals $v^{\top}A_\theta v$, and the left-hand side of~\eqref{l3eq} equals $\sup_{\theta\in\mathbb{R}^{d+1},\,v\in\mathbb{S}^d}|v^{\top}A_\theta v|$,  which, by Remark~\ref{rem:measurability}, is almost surely equal to a countable supremum and is therefore measurable.

Let $N_n\subset\mathbb{S}^d$ be a deterministic $1/4$-net of $\mathbb{S}^d$ with $|N_n|\leq9^{d+1}$, which exists by the covering-number bound \cite[Corollary~4.2.11]{Vershynin_2026}. Since each $A_\theta$ is symmetric, the net bound for quadratic forms \cite[Lemma~4.4.2]{Vershynin_2026} gives, for every $\theta\in\mathbb{R}^{d+1}$,
\[\sup_{v\in\mathbb{S}^d}|v^\top A_\theta v|\leq 2\max_{v\in N_n}|v^\top A_\theta v|.\] 
Taking the supremum over $\theta$ and interchanging it with the maximum over the finite set $N_n$ yields
\[ \sup_{\theta\in\mathbb{R}^{d+1},\,v\in\mathbb{S}^d}|v^\top A_\theta v| \leq 2\max_{v\in N_n}\sup_{\theta\in\mathbb{R}^{d+1}}|v^\top A_\theta v|.\]
The right-hand side is twice the quantity in~\eqref{eq:aux_pe}, so Lemma~\ref{lem:aux_pe} implies that it converges to zero almost surely, which proves~\eqref{l3eq}.
\end{proof}

\begin{lemma}\label{da1}
Under Assumption~\ref{as1}, it holds 
that~$\inf_{\theta\in\mathbb{R}^{d+1},v\in\mathbb{S}^d}\mathbb{E}[\mathds{1}_{(-\infty,0]}(y_1\bar{x}_1\cdot \theta)|\bar{x}_1\cdot v|^2] = m_{\beta_0,\gamma_0}^*$, where
\begin{equation}\label{mbgdef}
m_{\beta_0,\gamma_0}^* := 
\frac{h_0+h_2}{2}-\bigg(\frac{(h_0-h_2)^2}{4}+h_1^2\bigg)^{1/2} > 0,
\end{equation}
where~$h_i=\mathbb{E}[\Phi(-|\beta_0+\gamma_0Z|)Z^i]$ with~$Z\sim N(0,1)$.
\end{lemma}
\begin{remark}\label{imt}
\begin{enumerate}
\item \label{imt1}
The quantity~$m_{\beta_0,\gamma_0}^*$ has the alternate expression~$\inf_{(\bar{v}_0,\bar{v}_1,\bar{v}_2)\in\mathbb{S}^2}(\mathbb{E}[\Phi(-|\beta_0+\gamma_0Z|)|\bar{v}_0+\bar{v}_1Z|^2] + \mathbb{E}[\Phi(-|\beta_0+\gamma_0Z|)]\cdot \bar{v}_2^2)$, see the proof of Lemma~\ref{da1}. 
In the null case~$\beta_0=\gamma_0=0$, we have~$m_{0,0}^*=\inf(\frac{1}{2}\mathbb{E}[|\bar{v}_0+\bar{v}_1 Z|^2] + \frac{1}{2}\bar{v}_2^2) = \frac{1}{2}\inf(\bar{v}_0^2 +\bar{v}_1^2 +\bar{v}_2^2) = 
\frac{1}{2}$. By definition of~$\Phi$, we have~$m_{\beta_0,\gamma_0}^*\leq m_{0,0}^*= 1/2$ for all~$\beta_0,\gamma_0$.
\item 
By symmetry of~$Z\sim N(0,1)$ around~$0$, it is easy to see that~$m_{\beta_0,\gamma_0}^*$ is symmetric in~$\beta_0$ around~$0$.
\item
Analytical simplifications are possible when~$\beta_0=0$, or when~$\gamma_0=0$. For example, with similar arguments as in~\ref{imt1}, we have~$m_{\beta_0,0}^*=\Phi(-|\beta_0|)$.
\end{enumerate}
\end{remark}

\begin{proof}
First, consider the~$\theta=0$ case. Here for any~$v=(v_0,\dots,v_d)\in\mathbb{S}^d$ we have~$\mathbb{E}[\mathds{1}_{(-\infty,0]}(y_1\bar{x}_1\cdot \theta)|\bar{x}_1\cdot v|^2] = \mathbb{E}[|\bar{x}_1\cdot v|^2] 
=v_0^2 + \mathbb{E}[|x_1\cdot (v_1,\dots,v_d)|^2] = 1 
\geq m_{\beta_0,\gamma_0}^*$,
where the last inequality follows by~$\Phi(-z)\leq 1/2$ for all~$z\geq 0$. 
For~$\theta\neq 0$, it suffices to consider~$\theta\in\mathbb{S}^d$. 
Since for any~$\theta\in\mathbb{S}^d$, we have a.s. (in particular excluding~$\bar{x}_1\cdot\theta=0$) that
\begin{equation*}
\mathbb{P}(y_1\bar{x}_1\cdot \theta\leq 0|\bar{x}_1) = \begin{cases}
\Phi(\bar{x}_1\cdot \bar{\beta}) &\textrm{if }\bar{x}_1\cdot \theta <0,\\
\Phi(-\bar{x}_1\cdot \bar{\beta}) &\textrm{if }\bar{x}_1\cdot \theta >0,
\end{cases}
\end{equation*}
we have 
\begin{align*}
\inf_{\theta,v\in\mathbb{S}^d}\mathbb{E}[\mathds{1}_{(-\infty,0]}(y_1\bar{x}_1\cdot \theta)|\bar{x}_1\cdot v|^2] 
&= \inf_{\theta,v\in\mathbb{S}^d}\mathbb{E}[\Phi(-\textrm{sgn}( \bar{x}_1\cdot\theta)\bar{x}_1\cdot \bar{\beta})|\bar{x}_1\cdot v|^2] \\
&= \inf_{v\in\mathbb{S}^d} \mathbb{E}[\Phi(-|\bar{x}_1\cdot \bar{\beta}|)|\bar{x}_1\cdot v|^2],
\end{align*}
where the second equality uses that~$\Phi(-\textrm{sgn}(t)s)\geq\Phi(-|s|)$ for all~$s,t\in\mathbb{R}$, with equality almost surely for the choice~$\theta=\bar{\beta}/|\bar{\beta}|$ if~$\bar{\beta}\neq0$, while for~$\bar{\beta}=0$ both integrands equal~$\tfrac{1}{2}|\bar{x}_1\cdot v|^2$ for every~$\theta$.
Recall that we can assume w.l.o.g.\ that~$\bar{\beta} = (\beta_0,\gamma_0,0,\dots,0)$. 
Since~$\bar{x}_1=(1,x_1^{\top})^{\top}$ with~$x_1\sim N(0,I_d)$, denoting~$v=(v_0,\dots,v_d)$, we have for~$d\neq 1$ (and similarly for~$d=1$)
\begin{align*}
\mathbb{E}[\Phi(-|\bar{x}_1\cdot \bar{\beta}|)|\bar{x}_1\cdot v|^2] &= \mathbb{E}[\Phi(-|\beta_0+\gamma_0x_{11}|)|v_0+v_1x_{11}|^2]\\
&\quad+ \mathbb{E}[\Phi(-|\beta_0+\gamma_0x_{11}|)]\cdot|(v_2,\dots, v_d)|^2
\end{align*}
from which it follows that
\begin{align*}
&\inf_{v\in\mathbb{S}^d} \mathbb{E}[\Phi(-|\bar{x}_1\cdot \bar{\beta}|)|\bar{x}_1\cdot v|^2] \\
&\quad= \inf_{\bar{v}=(\bar{v}_0,\bar{v}_1,\bar{v}_2)\in\mathbb{S}^2}\big(\mathbb{E}[\Phi(-|\beta_0+\gamma_0Z|)|\bar{v}_0+\bar{v}_1Z|^2] + \mathbb{E}[\Phi(-|\beta_0+\gamma_0Z|)]\cdot \bar{v}_2^2\big)\\
&\quad= \lambda_{\min}(Q),
\end{align*}
where~$Z\sim N(0,1)$ and
\begin{align}
Q=\begin{pmatrix}\mathbb{E}[\Phi(-|\beta_0+\gamma_0 Z|)] & \mathbb{E}[\Phi(-|\beta_0+\gamma_0 Z|)Z] & 0\\
\mathbb{E}[\Phi(-|\beta_0+\gamma_0 Z|)Z] & \mathbb{E}[\Phi(-|\beta_0+\gamma_0 Z|)Z^2] & 0\\
0 & 0 & \mathbb{E}[\Phi(-|\beta_0+\gamma_0 Z|)]\end{pmatrix}.
\end{align}
The eigenvalues of~$Q$ are~$\mathbb{E}[\Phi(-|\beta_0+\gamma_0 Z|)]$ and those of the top left~$2$-by-$2$ submatrix~$Q'$, which are~$(h_0+h_2)/2\pm(1/2)((h_0-h_2)^2+4h_1^2)^{1/2}$ where~$h_i=\mathbb{E}[\Phi(-|\beta_0+\gamma_0 Z|)Z^i]$. Since~$\lambda_{\min}(Q') \leq \begin{pmatrix}1 & 0\end{pmatrix}Q'\begin{pmatrix}1 & 0\end{pmatrix}^{\top} = h_0$, the assertion follows.
\end{proof}

\begin{theorem}\label{nmt}
Let Assumption~\ref{as1} hold. Let~$\delta>0$. Let~$m_{\beta_0,\gamma_0}^*$ be given by~\eqref{mbgdef}. Then 
$$
\mathbb{P}(L/m
\leq (1+\delta)\pi/(2m_{\beta_0,\gamma_0}^*))\to 1
$$
as~$n,d\rightarrow \infty$ subject to~$\lim_{n,d\rightarrow\infty}d/n=0$.\\
\end{theorem}

\begin{proof}
By Lemmata~\ref{da2} and~\ref{da1}, and the lower bound in~\eqref{sps} for~$m$, we have for any~$\eta>0$ that
\begin{equation}\label{mfa}
\mathbb{P}(m/n\geq 2(m_{\beta_0,\gamma_0}^* - \eta)/\pi)\rightarrow 1.
\end{equation}
On the other hand, we have
\begin{equation*}
\sum_{i=1}^n\bar{x}_i\bar{x}_i^{\top} = \begin{pmatrix}
n & 0\\
0 & \sum_{i=1}^nx_ix_i^{\top}
\end{pmatrix}
+ 
\begin{pmatrix}
0 & \sum_{i=1}^nx_i^{\top}\\
\sum_{i=1}^nx_i & 0
\end{pmatrix},
\end{equation*}
where the second matrix on the right-hand side has operator norm $|\sum_{i=1}^nx_i|$, its nonzero eigenvalues being $\pm\,|\sum_{i=1}^nx_i|$. By Weyl's inequality and the properties~\eqref{probit}-\eqref{eq:def_m_L} of~$L$, it follows that 
\begin{equation}\label{lka}
\frac{L}{n} \leq \max\bigg(1,\lambda_{\max}\bigg(\frac{1}{n}\sum_{i=1}^nx_ix_i^{\top}\bigg)\bigg) + \bigg|\frac{1}{n}\sum_{i=1}^nx_i\bigg|.
\end{equation}
Fix~$\eta\in(0,m^*_{\beta_0,\gamma_0})$. Since~$x_i\sim N(0,I_d)$ by assumption and $d/n\rightarrow 0$, Theorem~6.1 in~\cite{MR3967104} gives
\begin{equation*}
\mathbb{P}(\lambda_{\max}(n^{-1}\textstyle\sum_{i=1}^nx_ix_i^{\top})\leq 1+\eta/2)\rightarrow 1.
\end{equation*}
Moreover, the coordinates of $n^{-1}\sum_{i=1}^nx_i$ are i.i.d.\ $N(0,1/n)$, so that $\mathbb{E}[|n^{-1}\sum_{i=1}^nx_i|^2]=d/n\rightarrow 0$ and hence $|n^{-1}\sum_{i=1}^nx_i|\leq\eta/2$ with probability tending to one.
Combining this with~\eqref{mfa} yields, with probability tending to one,
\[\frac{L}{m}=\frac{L/n}{m/n}\leq\frac{(1+\eta)\,\pi}{2(m^*_{\beta_0,\gamma_0}-\eta)}.\]
Since~$m^*_{\beta_0,\gamma_0}>0$, the right-hand side is at most $(1+\delta)\pi/(2m^*_{\beta_0,\gamma_0})$ for all sufficiently small~$\eta>0$, which concludes the proof.
\end{proof}

\begin{figure}[h]
  \centering
  \begin{subfigure}[h]{0.38\linewidth}
    \adjincludegraphics[trim={6mm 0 3mm 0},width=\linewidth]{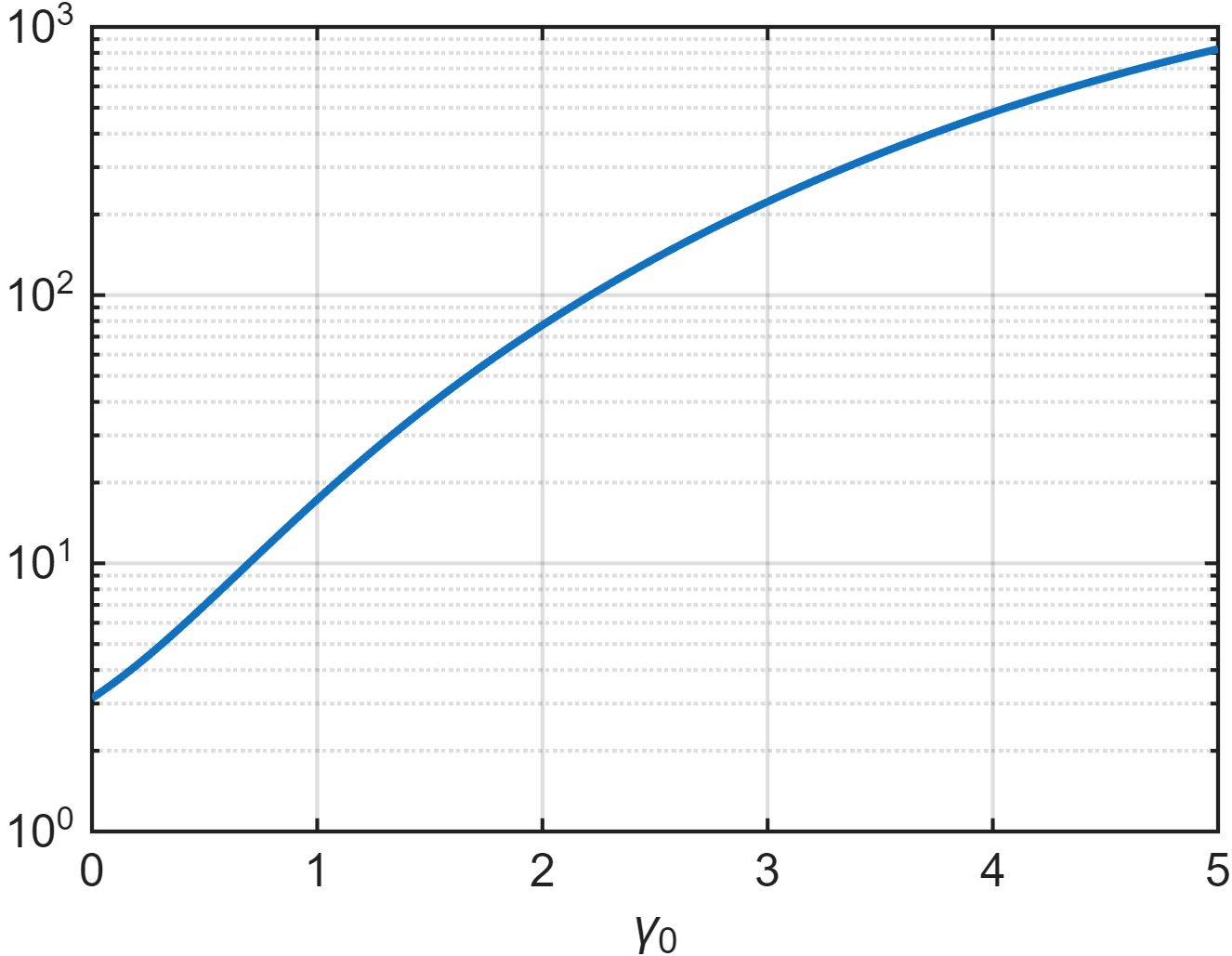}
  \end{subfigure}
  \hspace{0.03\linewidth}
    \begin{subfigure}[h]{0.48\linewidth}
    \adjincludegraphics[trim={6mm 0 3mm 0},width=\linewidth]{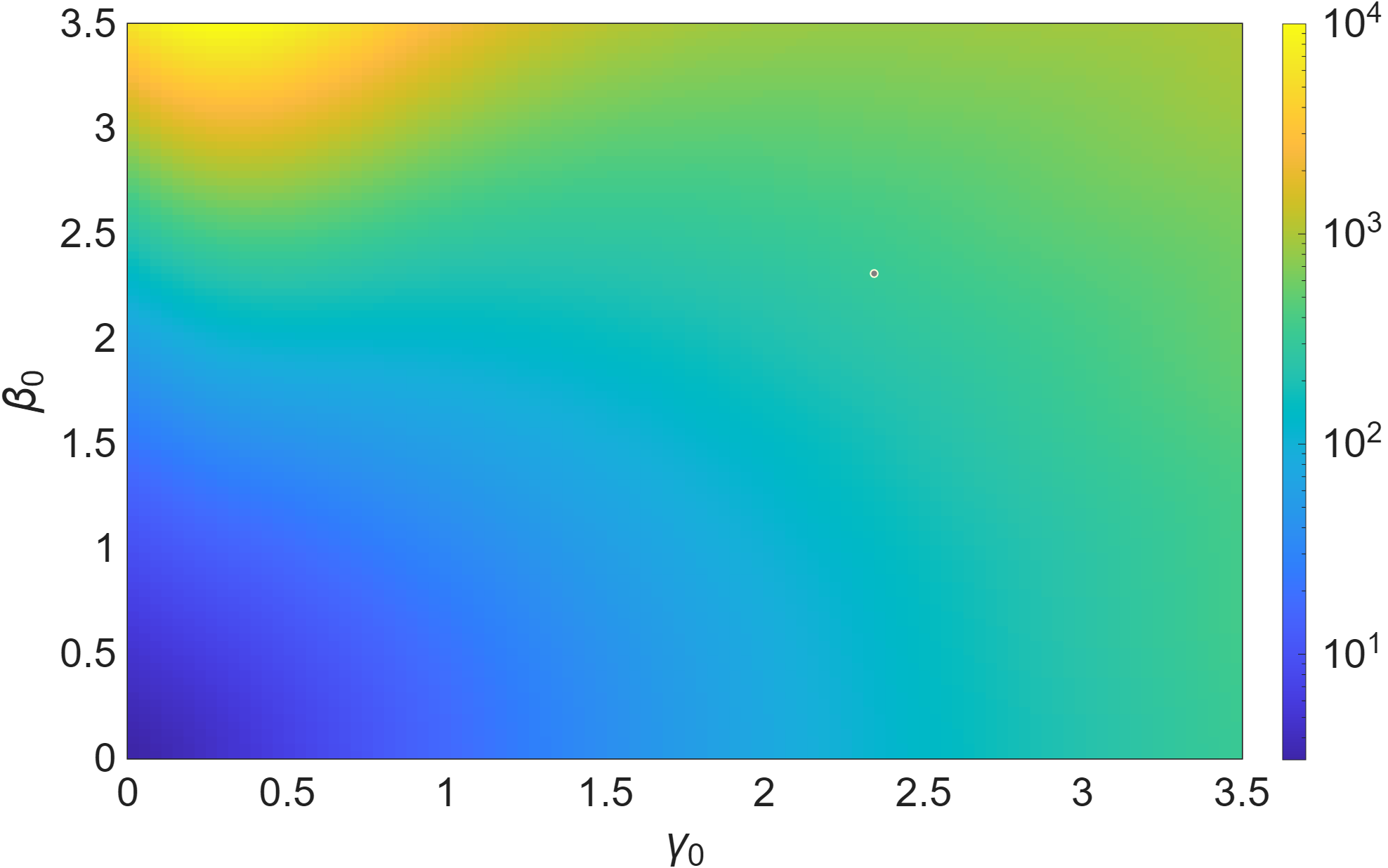}
  \end{subfigure}
\caption{Asymptotic ($n,d\to\infty$, $d/n\rightarrow0$) upper bound on condition number. Left: $\beta_0=0$ and varying~$\gamma_0$. Right: varying~$\beta_0$ and $\gamma_0$. Both plots represent~$\pi/(2m_{\beta_0,\gamma_0}^*)$ given by~\eqref{mbgdef}.
}\label{fig2}
\end{figure}

Our bounds complement standard results in convex optimization \cite{MR2142598} or log-concave sampling \cite{Chewi26Book}, which provide complexity guarantees for computing (possibly penalized) maximum likelihood estimators or sampling from Bayesian probit regression posteriors. In particular, these standard results have an explicit polynomial dependence on~$L/m$ (with exponent at least~$1/2$), so quantifying the scaling of~$L/m$ informs on how the required computational effort scales depending on the data (e.g. how it depends on $n$, $d$, $\beta_0$ and $\gamma_0$ for data generated according to Assumption~\ref{as1}).

The upper bound on $L/m$ in Theorem \ref{nmt} appears to be increasing with respect to $\beta_0$. On the other hand, monotonicity with respect to $\gamma_0$ is more subtle: see e.g.\ the behaviour near the region $(\beta_0,\gamma_0)=(3,0)$, where the bound first increases and then decreases with $\gamma_0$.
Numerical explorations, based on randomly generating covariates and data points according to Assumption~\ref{as1} and then approximating the resulting values of $L$ and $m$ using their definition in \eqref{eq:def_m_L} and numerical optimization over $\theta\in\mathbb{R}^{d+1}$, suggest that 
lack of monotonicity with respect to $\gamma_0$ 
is not an artefact of the upper bound but rather intrinsic to the actual condition number.

\paragraph{Acknowledgments}
MC and GZ acknowledge support from the European Research Council (ERC),
through StG “PrSc-HDBayLe” grant ID 101076564.
\bibliography{bibliography}
\end{document}